\documentclass[12pt,a4paper]{article}
\usepackage{latexsym,amssymb,amsfonts,amsmath,amsthm,nccmath,enumitem}
\usepackage[hidelinks]{hyperref}
\usepackage[margin=2cm]{geometry}
\usepackage[usenames,dvipsnames]{xcolor}

\usepackage{bm}

\numberwithin{equation}{section}
\newtheorem{theorem}{Theorem}

\newtheorem{lemma}[theorem]{Lemma}

\theoremstyle{definition}

\let\oldproofname=\proofname
\renewcommand{\proofname}{\rm\bf{\oldproofname}}

\def \D {{\mathcal{D}}}

\def \P {{\mathcal {P}}}

\def \a {\alpha}
\def \b {\beta}

\def \E {\varepsilon}

\def \mod#1{{\:({\rm mod}\ #1)}}

\def \deg {{\rm deg}}

\def \case#1 {\noindent {\bf Case #1.}\quad}

\def \floor#1{\lfloor #1 \rfloor}

\renewcommand{\leq}{\leqslant}
\renewcommand{\geq}{\geqslant}

\title{\bf Cycle packings of the complete multigraph}

\author{
Rosalind A. Cameron \\
Department of Mathematics and Statistics\\
Memorial University of Newfoundland\\
St John's, NL\\
\texttt{rahoyte@outlook.com}
}

\date{ }

\begin{document}
\maketitle\thispagestyle{empty}
\def\baselinestretch{1.1}\small\normalsize
\sloppy

\begin{abstract}
\noindent Bryant, Horsley, Maenhaut and Smith recently gave necessary and sufficient conditions for when the complete multigraph can be decomposed into cycles of specified lengths $m_1,m_2,\ldots,m_\tau$.
In this paper we characterise exactly when there exists a packing of the complete multigraph with cycles of specified lengths $m_1,m_2,\dots,m_\tau$. While cycle decompositions can give rise to packings by removing cycles from the decomposition, in general it is not known when there exists a packing of the complete multigraph with cycles of various specified lengths. 
\end{abstract}

\section{Introduction}\label{Section:Intro}

A \emph{decomposition} of a multigraph $G$ is a collection $\mathcal{D}$ of submultigraphs of $G$ such that each edge of $G$ is in {exactly one} of the multigraphs in $\mathcal{D}$. A \emph{packing} of a multigraph $G$ is a collection $\mathcal{P}$ of submultigraphs of $G$ such that each edge of $G$ is in {at most one} of the multigraphs in $\mathcal{P}$. The \emph{leave} of a packing $\P$ is the multigraph obtained by removing the edges in multigraphs in $\P$ from $G$. 
A  \emph{cycle packing} of a multigraph $G$ is a packing $\P$ of $G$ such that each submultigraph in $\P$ is a cycle. 
For positive integers $\lambda$ and $v$, $\lambda K_v$ denotes the complete multigraph with $\lambda$ parallel edges between each pair of $v$ distinct vertices.
Here we give a complete characterisation of when there exists a packing of $\lambda K_v$ with cycles of specified lengths $m_1,m_2,\dots,m_\tau$. Note that for $\lambda\geq 2$, $\lambda K_v$ contains $2$-cycles (pairs of parallel edges).

\begin{theorem}\label{Theorem:LambdaPackings}
Let $m_1,m_2,\ldots,m_\tau$ be a nondecreasing list of integers and let $\lambda$ and $v$ be positive integers.
Then there exists a packing of $\lambda K_v$ with cycles of lengths $m_1,m_2,\ldots,m_\tau$ if and only if 
\begin{itemize}
\item[(i)] $2\leq m_1\leq m_2,\ldots,m_\tau\leq v$;

\item[(ii)] $m_1+m_2+\dots+m_\tau=\lambda\binom{v}{2}-\delta$, where $\delta$ is a nonnegative integer such that $\delta\neq 1$ when $\lambda (v-1)$ is even, $\delta\neq 2$ when $\lambda=1$,  and $\delta\geq \frac{v}{2}$ when $\lambda (v-1)$ is odd; 

\item[(iii)] $\sum_{m_i=2}m_i\leq \left\{\begin{array}{l l}
(\lambda-1)\binom{v}{2}-2& \text{if $\lambda$ and $v$ are odd and $\delta=2$,}\\[1mm]
(\lambda-1)\binom{v}{2}& \text{if $\lambda$ is odd; and}
\end{array}
\right.$ 

\item[(iv)] $m_\tau\leq \left\{\begin{array}{ll}
\frac{\lambda}{2}\binom{v}{2}-\tau+2 & \text{if $\lambda$ is even and $\delta=0$},\\[1mm]
\frac{\lambda}{2}\binom{v}{2}-\tau+1 & \text{if $\lambda$ is even and $2\leq \delta<m_\tau$.}\\
\end{array}\right.$
\end{itemize}
\end{theorem}

Bryant, Horsley, Maenhaut and Smith  \cite{BHMS15} recently characterised exactly when there exists a decomposition of the complete multigraph $\lambda K_v$ into cycles of specified lengths $m_1,m_2,\dots,m_\tau$ (see also \cite{BHMS11,Smith10}). Since a decomposition of a multigraph is a packing whose leave contains no edges, many instances of the cycle packing problem can be solved by removing cycles from a cycle decomposition $\lambda K_v$. However there are cases which cannot be solved in this manner. These cases occur when $\lambda(v-1)$ is odd and there are $\tfrac{v}{2}+1$ or $\tfrac{v}{2}+2$ edges in the leave of the required packing.

In the case of the complete graph $K_v$ (with $\lambda=1$), it had previously been found exactly when there exist decompositions into cycles of specified lengths \cite{BryHorPet14}. Furthermore, Horsley \cite{Horsley11} found conditions for the existence of packings of the complete graph with uniform length cycles. These results built on earlier results for cycle decompositions and packings of the complete graph \cite{AlsGav01,Balister01_A,ElZanati94,Sajna02_ii} (see \cite{BryRod07} for a survey).  However, even in the  $\lambda =1$ case, a complete characterisation of when there exists a packing of $K_v$ with cycles of lengths $m_1,m_2,\dots,m_\tau$ had not previously been obtained.

We will show that the necessity of conditions (i)--(iv) in Theorem~\ref{Theorem:LambdaPackings} follows from known results for cycle decompositions of $\lambda K_v$. The sufficiency of these conditions is proved by first decomposing $\lambda K_v$ into cycles (and a $1$-factor if $\lambda (v-1)$ is even) and then removing cycles and modifying the resulting packing to obtain the one that we require. The existence of these cycle decompositions of $\lambda K_v$ was obtained by Bryant et al \cite{BHMS15} and the exact result is stated as Theorem~\ref{Theorem:MultigraphDecomp} in Section~\ref{Section:proof}. Section~\ref{Section:lambda_prelim} contains the results required for modifying cycle decompositions.

The following definitions and notation will be used throughout this paper. 
An $(m_1,m_2,\ldots,m_\tau)$-decomposition of $\lambda K_v$ is a decomposition of $\lambda K_v$ into $\tau$ cycles of lengths $m_1,m_2,\ldots,m_\tau$. Similarly, an $(m_1,m_2,\ldots,m_\tau)$-packing of $\lambda K_v$ is a packing of $\lambda K_v$ with $\tau$ cycles of lengths $m_1,m_2,\ldots,m_\tau$. 
We shall write $(m_1^{\ell_1},m_2^{\ell_2}\ldots,m_\tau^{\ell_\tau})$ to denote the list of integers $(\underbrace{m_1,\ldots,m_1}_{\ell_1},\underbrace{m_2,\ldots,m_2}_{\ell_2},\ldots,\underbrace{m_\tau,\ldots,m_\tau}_{\ell_\tau})$. 

For vertices $x$ and $y$ in a multigraph $G$, the \emph{multiplicity} of $xy$ is the number of edges in $G$ which have $x$ and $y$ as their endpoints, denoted $\mu_G(xy)$. If $\mu_G(xy)\leq 1$ for all pairs of vertices in $V(G)$ then we say that $G$ is a \emph{simple} graph. 
A multigraph is said to be \emph{even} if every vertex has even degree and is said to be \emph{odd} if every vertex has odd degree.

Given a permutation $\pi$ of a set $V$, a subset $S$ of $V$ and a multigraph $G$ with $V(G)\subseteq V$, $\pi(S)$ is defined to be the set $\{\pi(x):x \in S\}$ and $\pi(G)$ is defined to be the multigraph with vertex set $\pi(V(G))$ and edge set $\{\pi(x)\pi(y):xy\in E(G)\}$. The $m$-cycle with vertices $x_0,x_1,\ldots,x_{m-1}$ and edges $x_ix_{i+1}$ for $i\in\{0,\ldots,m-1\}$ (with subscripts modulo $m$) is denoted by $(x_0,x_1,\ldots,x_{m-1})$ and the $n$-path with vertices $y_0,y_1,\ldots,y_n$ and edges $y_jy_{j+1}$ for $j\in \{0,1,\ldots,n-1\}$ is denoted by $[y_0,y_1,\ldots,y_n]$. 

A \emph{chord} of a cycle is an edge which is incident with two vertices of the cycle but is not in the cycle. Note that a chord may be an edge parallel to an edge of the cycle.
For integers $p\geq 2$ and $q\geq 1$, a \emph{$(p,q)$-lasso} is the union of a $p$-cycle and a $q$-path such that the cycle and the path share exactly one vertex and that vertex is an end-vertex of the path.  A $(p,q)$-lasso with cycle $(x_1, x_2,\ldots,x_p)$ and path $[x_p, y_1, y_2,\ldots,y_q]$ is denoted by  $(x_1, x_2,\ldots,x_p)[x_p, y_1, y_2,\ldots,y_q]$. The \emph{order} of a $(p,q)$-lasso is $p+q$.

\section{Modifying cycle packings of $\lambda K_v$}\label{Section:lambda_prelim}

The aim of this section is to prove Lemmas~\ref{Lemma:LassoToCycle_lambda} and \ref{Lemma:ChordToLasso_lambda}. These results are useful tools for modifying cycle packings of the complete multigraph. 
The simple graph versions of Lemmas~\ref{Lemma:LassoToCycle_lambda} and \ref{Lemma:ChordToLasso_lambda} are due to Bryant and Horsley \cite{BryHor08} and have been applied to prove the maximum packing result of the simple complete graph with uniform length cycles \cite{Horsley11}.

We require the following cycle switching lemma for cycle packings of multigraphs. Lemma~\ref{Lemma:MultiCycleSwitch} is similar to \cite[Lemma 2.1]{BHMS11} and is also closely related to the cycle switching method which has been applied to simple graphs (see for example \cite{BryHor09}).

\begin{lemma}\label{Lemma:MultiCycleSwitch}
Let $v$ and $\lambda$ be positive integers, let $M$ be a list of integers, let $\P$ be an $(M)$-packing of $\lambda K_v$, let $L$ be the leave of $\P$,  let $\a$ and $\b$ be distinct vertices of $L$, and let $\pi$ be the transposition $(\a\b)$. 
Let $E$ be a subset of $E(L)$ such that, for each vertex $x\in V(L)\setminus\{\a,\b\}$, $E$ contains precisely $\max(0,\mu_L(x\a)-\mu_L(x\b))$ edges with endpoints $x$ and $\a$, and precisely $\max(0,\mu_L(x\b)-\mu_L(x\a))$ edges with endpoints $x$ and $\b$ (so $E$ may contain multiple edges with the same endpoints), and $E$ contains no other edges. 
Then there exists a partition of $E$ into pairs such that for each pair $\{x_1y_1,x_2y_2\}$ of the partition, there exists an $(M)$-packing $\mathcal{P}'$ of $\lambda K_v$ with leave $L'=(L-\{x_1y_1,x_2y_2\})+\{\pi(x_1)\pi(y_1),\pi(x_2)\pi(y_2)\}$. 

Furthermore, if $\P=\{C_1,\ldots,C_t\}$, then $\P'=\{C'_1,\ldots,C'_t\}$ where for $i\in\{1,\ldots,t\}$, $C'_i$ is a cycle of the same length as $C_i$ such that for $i\in\{1,\ldots,t\}$
\begin{itemize}
\item
If neither $\a$ nor $\b$ is in $V(C_i)$, then $C'_i=C_i$;

\item
If exactly one of $\a$ and $\b$ is in $V(C_i)$, then $C'_i=C_i$ or $C'_i=\pi(C_i)$; and

\item
If both $\a$ and $\b$ are in $V(C_i)$, then $C'_i=Q_i\cup Q_i^*$ where $Q_i=P_i$ or $\pi(P_i)$, $Q_i^*=P_i^*$ or $\pi(P_i^*)$, and $P_i$ and $P_i^*$ are the two paths from $\a$ to $\b$ in $C_i$. 
\end{itemize}
\end{lemma}

\begin{proof}
When $\lambda (v-1)$ is even, Lemma~\ref{Lemma:MultiCycleSwitch} reduces to \cite[Lemma 2.1]{BHMS11} so suppose $\lambda (v-1)$  is odd.
Note that $\P$ is a cycle packing of $\lambda K_v$ regardless of the parity of $\lambda(v-1)$, whereas when $\lambda (v-1)$  is odd \cite[Lemma 2.1]{BHMS11} concerns a cycle packing of $\lambda K_v-I$, where $I$ is a $1$-factor of $\lambda K_v$. Nevertheless, the proof of Lemma~\ref{Lemma:MultiCycleSwitch} follows from similar arguments to those used in the corresponding case of the  proof in \cite{BHMS11}.
\end{proof}

In applying Lemma~\ref{Lemma:MultiCycleSwitch} we say that we are performing the $(\a,\b)$-switch with origin $x$ and terminus $y$ (where $\{x_1,y_1,x_2,y_2\}\subseteq \{\a,\b,x,y\}$). Note that $x_1y_1$ and $x_2y_2$ may be parallel edges, in which case $x=y$.

\begin{lemma}\label{Lemma:LassoToCycle_lambda}\label{Lemma:LassoToCycle}
Let $v$, $s$ and $\lambda$ be positive integers such that $s\geq 3$, and let $M$ be a list of integers. Suppose there exists an $(M)$-packing $\P$ of $\lambda K_v$ whose leave contains a lasso of order at least $s+2$ and suppose that if $s$ is even then the cycle of the lasso has even length. Then there exists an $(M,s)$-packing of $\lambda K_v$.
\end{lemma}

\begin{proof}
Let $L$ be the leave of $\P$. Suppose that $L$ contains a $(p,q)$-lasso $(x_1, x_2,\ldots,x_p)[x_p, y_1, y_2,\ldots,y_q]$ such that $p+q\geq s+2$ and $p$ is even if $s$ is even.  
If $L$ contains an $s$-cycle then we add it to the packing to complete the proof, so assume $L$ does not contain an $s$-cycle and hence $p\neq s$.

\noindent {\bf Case 1.} Suppose $2\leq p<s$ and either $p=2$ or $p\equiv s\mod{2}$. We can assume that $p+q=s+2$ since $L$ contains a $(p,s+2-p)$-lasso. 

Let  $L'$ be the leave of the packing $\P'$ obtained from $\P$ by applying the $(x_1,y_{q-1})$-switch with origin $x_2$ (note that $\mu_L(x_2y_{q-1})=0$ for otherwise $L$ contains an $s$-cycle). 
If the terminus of the switch is not $y_{q-2}$ then $L'$ contains an $s$-cycle which completes the proof (recall that $s=p+q-2$). Otherwise, the terminus of the switch is $y_{q-2}$ and $L'$ contains a $(q,p)$-lasso $(x'_1, x'_2,\ldots,x'_q)[x'_q, y'_1, y'_2,\ldots,y'_p]$. If $p=2$ then $L'$ contains an $s$-cycle which completes the proof, so assume $L'$ contains no $s$-cycle and $p\geq 3$.

Let $L''$ be the leave of the packing $\P''$ obtained from $\P'$ by applying the $(x'_2,y'_p)$-switch with origin $x'_3$ (note that $\mu_{L'}(x'_3y'_{p})=0$ for otherwise $L'$ contains an $s$-cycle). If the terminus of this switch is not $y'_{p-1}$ then $L''$ contains an $s$-cycle which completes the proof (recall that $s=p+q-2$). Otherwise, the terminus of the switch is $y'_{p-1}$ and $L''$ contains a $(p+2,q-2)$-lasso, so since $p<s$ and $p\equiv s\mod{2}$, the result follows by repeating the procedure described in this case.

\noindent {\bf Case 2.} Suppose $3\leq p<s$ and $p\not\equiv s\mod{2}$. As above, assume $p+q=s+2$. Then $s$ is odd, $p\geq 4$ is even and $q$ is odd by our hypotheses. 

Let $L'$ be the leave of the packing $\P'$ obtained from $\P$ by applying the $(x_2,y_q)$-switch with origin $x_3$ (note that $\mu_{L}(x_3y_{q})=0$ for otherwise $L$ contains an $s$-cycle). If the terminus of the switch is not $y_{q-1}$ then $L'$ contains an $s$-cycle which completes the proof. 
Otherwise, the terminus of the switch is $y_{q-1}$ and $L'$ contains a $(q+2,p-2)$-lasso. Note that $q+2\leq s$ (because $p+q=s+2$ and $p\geq 4$) and $q+2\equiv s\mod{2}$. If $q+2=s$ then this completes the proof, otherwise we can proceed as in Case 1.

\noindent {\bf Case 3.} Suppose $3\leq s< p$. 
Let $L'$ be the leave of the packing $\P'$ obtained from $\P$ by applying the $(x_{p-s+1},y_1)$-switch with origin $x_{p-s+2}$ (note that $\mu_L(x_{p-s+2}y_1)=0$ for otherwise $L$ contains an $s$-cycle). 
If the terminus of the switch is not $x_p$ then $L'$ contains an $s$-cycle which completes the proof. Otherwise, $L'$ contains a $(p-s+2,q+s-2)$-lasso. By repeating this process we obtain an $(M)$-packing of $\lambda K_v$ whose leave contains a $(p',p+q-p')$-lasso such that $2\leq p'\leq s$ and $p'\equiv p\mod{(s+2)}$. If $p'=s$ then this completes the proof, otherwise we can proceed as in Case 1 or Case 2. 
\end{proof}

\begin{lemma}\label{Lemma:ChordToLasso_lambda}\label{Lemma:ChordToLasso}
Let  $v$, $s$ and $\lambda$ be positive integers with $s\geq 3$, and let $M$ be a list of integers. Suppose there exists an $(M)$-packing of $\lambda K_v$ whose leave $L$ has a component $H$ containing an $(s+1)$-cycle with a chord. Then there exists an $(M)$-packing of $\lambda K_v$ with a leave $L'$ such that $E(L')=(E(L)\setminus E(H))\cup E(H')$, where $H'$ is a graph with $V(H')=V(H)$ and $|E(H')|=|E(H)|$ which contains an $(s,1)$-lasso. Furthermore,  $\deg_{H'}(x)\geq \deg_H(x)$ for each vertex $x$ in the $s$-cycle of this lasso.
\end{lemma}
 
  \begin{proof}
Let $(x_1,\ldots,x_{s+1})$ be an $(s+1)$-cycle in $H$ with chord $x_1x_e$ for some $e\in\{2,3,\ldots,s-1\}$ (note that $L$ is not necessarily a simple graph). If $H$ contains an $(s,1)$-lasso then we are finished immediately, so suppose otherwise. 
If ${e=2}$, then perform the $(x_3,x_2)$-switch with origin $x_4$ (note that $\mu_L(x_2x_4)=0$ because $H$ contains no $(s,1)$-lasso).  
The leave of the resulting packing contains the $(s,1)$-lasso $(x_4,\ldots,x_{s+1},x_1,x_2)[x_2,x_3]$, and $\deg_{H'}(x_i)\geq\deg_H(x_i)$ for $i\in \{1,\ldots,s+1\}\setminus\{3\}$. If $e=3$, then $H$ contains an $(s,1)$-lasso which completes the proof. 

So suppose $e\geq 4$ and let $\P^*$ be the packing with leave $L^*$ obtained from $\P$ by applying the $(x_{e-1},x_e)$-switch with origin $x_{e-2}$ (note that $\mu_L(x_{e-2}x_e)=0$ for otherwise $L$ contains an $(s,1)$-lasso). If the terminus of the switch is not $x_{e+1}$ then $E(L^*)=(E(L)\setminus E(H))\cup E(H^*)$, where $H^*$ is a graph with $V(H^*)=V(H)$ and $|E(H^*)|=|E(H)|$ which contains the {$(s,1)$-lasso} $(x_{e+1},\ldots,x_{s+1},x_1,\ldots,x_{e-2},x_{e})[x_e,x_{e-1}]$. Also note that 
$\deg_{H^*}(x_e)\geq \deg_H(x_e)$ and $\deg_{H^*}(x_i)=\deg_H(x_i)$ for $i\in \{1,\ldots,s+1\}\setminus\{e,e-1\}$. 
Otherwise, the terminus of the switch is $x_{e+1}$ and $E(L^*)=(E(L)\setminus E(H))\cup E(H^*)$, where $H^*$ is a graph with $V(H^*)=V(H)$ and $|E(H^*)|=|E(H)|$ which contains an $(s+1)$-cycle $(x^*_1,\ldots,x^*_{s+1})$ with chord $x^*_1x^*_{e-1}$. Furthermore, the degree of each vertex in this $(s+1)$-cycle remains unchanged in $H'$.
The result follows by repeating this process.
  \end{proof}

\section{Main result}\label{Section:proof}

This section contains the proof of Theorem~\ref{Theorem:LambdaPackings}. 
We first use Theorem~\ref{Theorem:MultigraphDecomp} (stated below) to prove Lemma~\ref{Lemma:lambda_necessaryconditions} which shows the necessity of the conditions in Theorem~\ref{Theorem:LambdaPackings}.
The sufficiency of these conditions is then established for $\lambda$ odd and $\lambda$ even in Lemmas~\ref{Lemma:lambda_odd} and \ref{Lemma:lambda_even} respectively. 
Lemmas~\ref{Lemma:lambda_odd} and \ref{Lemma:lambda_even} rely on using Lemmas~\ref{Lemma:LassoToCycle_lambda} and \ref{Lemma:ChordToLasso_lambda} to modify cycle packings of $\lambda K_v$ obtained via Theorem~\ref{Theorem:MultigraphDecomp}.

\begin{theorem}[{\cite{BHMS15}}]\label{Theorem:MultigraphDecomp} Let $m_1,m_2,\ldots,m_\tau$ be a nondecreasing list  of integers and let $\lambda$ and $v$ be positive integers.
There is an $(m_1,m_2,\ldots,m_\tau)$-decomposition of $\lambda K_v$ if and only if
\begin{itemize}
\item $\lambda(v-1)$ is even; 
\item $2\leq m_1\leq m_2,\ldots,m_\tau\leq v$;
\item $m_1+m_2+\cdots+m_\tau=\lambda\binom{v}{2}$;
\item $m_\tau+\tau -2\leq \frac{\lambda}{2}\binom{v}{2}$ when $\lambda$ is even; and
\item $\sum_{m_i=2}m_i\leq (\lambda-1)\binom{v}{2}$ when $\lambda$ is odd.
\end{itemize}
There is an $(m_1,m_2,\ldots,m_\tau)$-decomposition of $\lambda K_v-I$, where $I$ is a $1$-factor in $\lambda K_v$, if and only if
\begin{itemize}
\item $\lambda(v-1)$ is odd;
\item $2\leq m_1\leq m_2,\ldots,m_\tau\leq v$;
\item $m_1+m_2+\cdots+m_\tau=\lambda\binom{v}{2}-\frac{v}{2}$; and
\item $\sum_{m_i=2}m_i\leq (\lambda-1)\binom{v}{2}$.
\end{itemize}
\end{theorem}

The necessity of conditions Theorem~\ref{Theorem:LambdaPackings}(i)--(iv) follows from Theorem \ref{Theorem:MultigraphDecomp}  as we now show.

\begin{lemma}\label{Lemma:lambda_necessaryconditions}
Let $m_1,m_2,\ldots,m_\tau$ be a nondecreasing list of integers  and let $\lambda$ and $v$ be positive integers.
If there exists an $(m_1,m_2,\ldots,m_\tau)$-packing of $\lambda K_v$ then
\begin{itemize}
\item[(i)] $2\leq m_1\leq m_2,\ldots,m_\tau\leq v$;

\item[(ii)] $m_1+m_2+\cdots+m_\tau=\lambda\binom{v}{2}-\delta$, where $\delta$ is a nonnegative integer such that $\delta\neq 1$ when $\lambda (v-1)$ is even, $\delta\neq 2$ when $\lambda=1$,  and $\delta\geq \frac{v}{2}$ when $\lambda (v-1)$ is odd; 

\item[(iii)] $\sum_{m_i=2}m_i\leq \left\{\begin{array}{l l}
(\lambda-1)\binom{v}{2}-2& \text{if $\lambda$ and $v$ are odd and $\delta=2$,}\\[1mm]
(\lambda-1)\binom{v}{2}& \text{if $\lambda$ is odd; and}
\end{array}
\right.$ 

\item[(iv)] $m_\tau\leq \left\{\begin{array}{ll}
\frac{\lambda}{2}\binom{v}{2}-\tau+2 & \text{if $\lambda$ is even and $\delta=0$},\\[1mm]
\frac{\lambda}{2}\binom{v}{2}-\tau+1 & \text{if $\lambda$ is even and $2\leq \delta<m_\tau$.}\\
\end{array}\right.$
\end{itemize}
\end{lemma}

\begin{proof} Suppose there exists an $(m_1,m_2,\ldots,m_\tau)$-packing $\P$ of $\lambda K_v$ with leave $L$. 
Condition (i) is obvious.
The degree of each vertex in $\lambda K_v$ is $\lambda (v-1)$, so if $\lambda (v-1)$ is even then $L$ is an even multigraph and if $\lambda (v-1)$ is odd then $L$ is an odd multigraph. Hence (ii) follows because an even graph cannot have a single edge, an even simple graph cannot have two edges, and an odd graph on $v$ vertices has at least $\frac{v}{2}$ edges. 
To see that condition (iii) holds, note that there are at most $\floor{\frac{\lambda}{2}}\binom{v}{2}$ edge-disjoint $2$-cycles in $\lambda K_v$. Furthermore, note that if $\lambda$ and $v$ are both odd and $\delta=2$ then $L$ is a $2$-cycle (because $L$ is an even multigraph and has two edges). 
If $\lambda$ is even and $\delta=0$ then (iv) follows directly from Theorem~\ref{Theorem:MultigraphDecomp}, so suppose $\lambda$ is even and $2\leq \delta<m_\tau$. Then $L$ contains at least one cycle so there exists an $(m_1,m_2,\ldots,m_\tau,M)$-decomposition of $\lambda K_v$ for some list $M$ containing at least one entry. So (iv) follows from Theorem~\ref{Theorem:MultigraphDecomp}. 
\end{proof}

It remains to prove the sufficiency of Theorem~\ref{Theorem:LambdaPackings}(i)--(iv) for the existence of cycle packings of $\lambda K_v$.

\begin{lemma}\label{Lemma:lambda_odd}
Let $m_1,m_2,\ldots,m_\tau$ be a nondecreasing list  of integers and let $\lambda$ and $v$ be positive integers with $\lambda$ odd. Then there exists an $(m_1,m_2,\ldots,m_\tau)$-packing of $\lambda K_v$ if and only if 
\begin{itemize}
\item[(i)] $2\leq m_1\leq m_2,\ldots,m_\tau\leq v$;

\item[(ii)] $m_1+m_2+\cdots+m_\tau=\lambda\binom{v}{2}-\delta$, where $\delta$ is a nonnegative integer such that $\delta\neq 1$, $(\lambda,\delta)\neq (1,2)$, and if $v$ is even then $\delta\geq \frac{v}{2}$; and

\item[(iii)] $\sum_{m_i=2}m_i\leq \left\{\begin{array}{l l}
(\lambda-1)\binom{v}{2}-2& \text{if $v$ is odd and $\delta=2$,}\\[1mm]
(\lambda-1)\binom{v}{2}& \text{otherwise.}
\end{array}
\right.$ 
\end{itemize}
\end{lemma}

\begin{proof} If there exists an $(m_1,m_2,\ldots,m_\tau)$-packing $\P$ of $\lambda K_v$, then conditions (i)--(iii) hold by Lemma~\ref{Lemma:lambda_necessaryconditions}. 
So it remains to show that if $\lambda$, $v$ and $m_1,m_2,\ldots,m_\tau$ satisfy (i)--(iii), then there is an $(m_1,m_2,\ldots,m_\tau)$-packing of $\lambda K_v$.

 Let $\E=\delta$ if $v$ is odd, and $\E=\delta-\frac{v}{2}$ if $v$ is even. If $\E=0$ then the result follows by Theorem~\ref{Theorem:MultigraphDecomp}. 
 If $v=2$, then $\E$ is even by (i) and (ii) and there exists a $2$-cycle decomposition of $\lambda K_2-I$, where $I$ is a $1$-factor of $\lambda K_2$, so the result follows. 
 So suppose $\E\geq 1$ and $v\geq 3$, and note that if $v$ is odd then $\E\neq 1$  and $(\lambda,\E)\neq (1,2)$.

\noindent {\bf Case 1.} Suppose $v$ is odd or $\E\geq 3$. Note that if $v$ is odd and $\E=2$ then $2+\sum_{m_i=2}m_i\leq (\lambda-1)\binom{v}{2}$ by (iii). 

We show that there exists a list 
$N$ such that $2\leq n\leq v$ for all $n\in N$, $\sum N=\E$ and $\sum_{n\in N, n=2}n + \sum_{m_j=2}m_j\leq (\lambda-1)\binom{v}{2}$. If this list exists, then by Theorem~\ref{Theorem:MultigraphDecomp} there exists an $(m_1,m_2,\ldots,m_\tau, N)$-decomposition $\D$ of $\lambda K_v$ (if $v$ is odd) or $\lambda K_v-I$ (if $v$ is even), where $I$ is a $1$-factor of $\lambda K_v$. We obtain the required packing by removing cycles of lengths $N$ from $\D$. 

We first consider $v=3$. 
If $v=3$ and $\E$ is even, then $m_i=3$ for some $i\in\{1,\ldots,\tau\}$ by (i) and (ii). Then $\E+\sum_{m_i=2}m_i\leq (\lambda-1)\binom{v}{2}$ by (ii) and we take $N=(2^{\E/2})$.
If $v=3$ and $\E$ is odd then $\E-3+\sum_{m_i=2}m_i\leq (\lambda-1)\binom{v}{2}$ by (ii) and we take $N=(2^{(\E-3)/2},3)$. 
In each of these cases we can see that there exists an $(m_1,m_2,\ldots,m_\tau,N)$-decomposition of $\lambda K_v$ since the assumptions of Theorem~\ref{Theorem:MultigraphDecomp} are satisfied by (i)--(iii).

Now assume $v\geq 4$ and let $q$ and $r$ be nonnegative integers such that $\E=vq+r$ and $0\leq r<v$. If $q=0$ or $r\not\in\{1,2\}$ then we take $N=(r,v^{q})$. 
If $q\geq 1$ and $r\in\{1,2\}$ then $N=(3,v-3+r,v^{q-1})$ (note that either $v-3+r\geq 3$, or $v=4$ and $r=1$). 
If $\E=2$ or $(v,r)=(4,1)$, then $N$ contains exactly one entry equal to $2$ and otherwise $n\geq 3$ for all $n\in N$. By (iii) and the assumptions of this case, if $\E=2$ then  $2+\sum_{m_i=2}m_i\leq (\lambda-1)\binom{v}{2}$. Further, if $v=4$ and $\E=4q+1$ for some $q\geq 1$ then (i) and (ii) imply that $m_i=3$ for some $i\in \{1,\ldots,\tau\}$ so again $2+\sum_{m_i=2}m_i\leq (\lambda-1)\binom{v}{2}$.
We can therefore see that there exists an $(m_1,m_2,\ldots,m_\tau,N)$-decomposition of $\lambda K_v$ (or $\lambda K_v-I$) since the assumptions of Theorem~\ref{Theorem:MultigraphDecomp} are satisfied by (i)--(iii) and the fact that $\sum N=\E$.

\noindent {\bf Case 2.} Suppose $v$ is even and $\E\in\{1,2\}$. 
Let $M=m_1,m_2,\ldots,m_\tau$ and let $m$ be the least odd entry in $M$ if $M$ contains an odd entry, otherwise let $m$ be the least entry in $M$ such that $m\geq 4$ (such an entry exists by (iii)).
Note that  $v\geq 4$ and if $\E=1$ then it follows from (ii) that $M$ contains an odd entry and $m$ is odd. 

\noindent {\bf Case 2a.} Suppose $m+\E\leq v$.
By Theorem~\ref{Theorem:MultigraphDecomp} there exists an $(M\setminus (m), m+\E)$-decomposition $\D$ of $\lambda K_v-I$, where $I$ is a $1$-factor of $\lambda K_v$. Let $\P$ be the $(M\setminus (m))$-packing of $\lambda K_v$ that is obtained by removing an $(m+\E)$-cycle from $\D$.  
Let $L$ be the leave of $\P$ and note that $L$ consists of an $(m+\E)$-cycle and the $1$-factor $I$. 

If $L$ contains an $(m+\E,1)$-lasso  then we apply Lemma~\ref{Lemma:LassoToCycle_lambda} to $\P$ with $s=m$ to complete the proof. The assumptions of Lemma~\ref{Lemma:LassoToCycle_lambda} are satisfied because $\E+1\geq 2$, and if $m$ is even then $M$ contains no odd entries so $\E=2$ by (ii).

 So suppose $L$ does not contain an  $(m+\E,1)$-lasso. Then $m+\E$ is even and $L$ contains a component $H$ such that $H$ is the union of an $(m+\E)$-cycle and a $1$-factor on $V(H)$. We apply Lemma~\ref{Lemma:ChordToLasso_lambda} to $\P$ with $s=m+\E-1$ to obtain an $(M\setminus (m))$-packing $\P'$ of $\lambda K_v$ whose leave $L'$ contains a component $H'$ on $m+\E$ vertices that has $\frac{3}{2}(m+\E)$ edges and contains an $(m+\E-1,1)$-lasso. 
  If $\E=1$ then we can add the $m$-cycle of this lasso to $\P'$ to complete the proof. 
  Otherwise $\E=2$ and $H'$ contains an $(m+1)$-cycle with a chord because $m\geq 3$ and any vertex in this cycle has degree at least $3$ (note that $\deg_H(x)=3$ for each $x\in V(H)$).
Then we can apply Lemma~\ref{Lemma:ChordToLasso_lambda} with $s=m$ to $\P'$ to obtain an $(M\setminus (m))$-packing $\P''$ of $\lambda K_v$ whose leave contains an $(m,1)$-lasso. We add the $m$-cycle of this lasso to $\P''$ to complete the proof. 

\noindent {\bf Case 2b.} Suppose $m+\E>v$. Then $m\geq v-1$ and $\E=2$ (note that $\E$ is even if $m=v$).

If $m=v$ then $m_i\in\{2,v\}$ for all $i\in\{1,\ldots,\tau\}$, so $\lambda \binom{v}{2}-\frac{v}{2}\equiv {2+\sum_{m_i=2}m_i \mod{v}}$ by (ii) and hence ${2+\sum_{m_i=2}m_i}\leq (\lambda-1)\binom{v}{2}$ by (iii). 
Then by Theorem~\ref{Theorem:MultigraphDecomp} there exists an $(M,2)$-decomposition $\D$ of $\lambda K_v-I$. We remove a $2$-cycle from $\D$ to complete the proof.

So suppose that $m=v-1$. 
Since $\E$ is even, $M$ contains an even number of odd entries, so at least two entries of $M$ are equal to $v-1$. 
Let $\D_0$ be an $(M\setminus ((v-1)^2),v^2)$-decomposition of $\lambda K_v-I$ which exists by Theorem~\ref{Theorem:MultigraphDecomp}. Let $\P_0$ be the $(M\setminus ((v-1)^2),v)$-packing of $\lambda K_v$ formed by removing a $v$-cycle from $\D_0$. The leave $L_0$ of $\P_0$ is the union of a $v$-cycle and the $1$-factor $I$. 
Let $\P_1$ be the packing obtained by applying Lemma~\ref{Lemma:ChordToLasso} to $\P_0$ with $s=v-1$. Then the leave of $\P_1$ contains a $(v-1,1)$-lasso. We add the $(v-1)$-cycle of this lasso to $\P_1$ and remove a $v$-cycle to obtain an $(M\setminus (v-1))$-packing $\P_2$ of $\lambda K_v$. The leave of $\P_2$ has size $3\tfrac{v}{2}+1$.

By applying Lemma~\ref{Lemma:ChordToLasso} to $\P_2$ with $s=v-1$ we obtain an $(M\setminus (v-1))$-packing $\P_3$ of $\lambda K_v$ whose leave contains a $(v-1,1)$-lasso. We add the $(v-1)$-cycle of this lasso to $\P_3$ to complete the proof.
\end{proof}

\begin{lemma}\label{Lemma:lambda_even}
Let $m_1,m_2,\ldots,m_\tau$ be a nondecreasing list of integers  and let $\lambda$ and $v$ be positive integers with $\lambda$ even. Then there exists an $(m_1,m_2,\ldots,m_\tau)$-packing of $\lambda K_v$ if and only if 
\begin{itemize}
\item[(i)] $2\leq m_1\leq m_2,\ldots,m_\tau\leq v$;
\item[(ii)] $m_1+m_2+\cdots+m_\tau = \lambda\binom{v}{2}-\delta$, where $\delta$ is a nonnegative integer such that $\delta\neq 1$; and
\item[(iii)] $m_\tau\leq \left\{\begin{array}{ll}
\frac{\lambda}{2}\binom{v}{2}-\tau+2 & \text{if } \delta=0,\\[1mm]
\frac{\lambda}{2}\binom{v}{2}-\tau+1 & \text{if } 2\leq \delta<m_\tau.\\
\end{array}\right.$
\end{itemize}
\end{lemma}

\begin{proof} 
If there exists an $(m_1,m_2,\ldots,m_\tau)$-packing $\P$ of $\lambda K_v$ with leave $L$, then conditions (i)--(iii) hold by Lemma~\ref{Lemma:lambda_necessaryconditions}. 
So it remains to show that if $\lambda$, $v$ and $m_1,m_2,\ldots,m_\tau$ satisfy (i)--(iii), then there exists an $(m_1,m_2,\ldots,m_\tau)$-packing of $\lambda K_v$. If $\delta=0$ then the result follows immediately from Theorem~\ref{Theorem:MultigraphDecomp}, so suppose $\delta\geq 2$.

Let \[N = \left\{ \begin{array}{l l}
(\delta) & \text{if } 2\leq \delta< m_\tau,\\ 
(2^{(\delta-m_\tau)/2},m_\tau) & \text{if } \delta\geq m_\tau \text{ and } \delta\equiv m_\tau\mod{2},\\
(2^{(\delta-m_\tau+1)/2},m_\tau-1) & \text{if } \delta\geq  m_\tau \text{ and }\delta\not\equiv m_\tau\mod{2}.\\
\end{array}
\right.\]
Note that in each case $\sum N=\delta$. We show that there exists an $(m_1,m_2,\ldots,m_\tau,N)$-decomposition $\D$ of $\lambda K_v$ because the assumptions of Theorem~\ref{Theorem:MultigraphDecomp} are satisfied by (i)--(iii) and the definition of $N$.  The required packing is then obtained by removing cycles of lengths $N$ from $\D$. 

Let $s$ be the number of entries in $N$. Let $M=m_1,m_2,\ldots,m_\tau$.
First observe that $\sum M+\sum N=\lambda\binom{v}{2}$ by (ii) and since $\sum N=\delta$.
By (i) and the definition of $N$ it also holds that $2\leq n\leq m_\tau\leq v$ for all $n\in N$. If $2\leq \delta<m_\tau$, then $m_\tau\leq \frac{\lambda}{2}\binom{v}{2}-\tau-s+2$ by (iii) and since $s=1$. If $\delta\geq m_\tau$, then because $\sum M\geq m_\tau+2(\tau-1)$ and $\sum N\geq m_\tau-1+2(s-1)$, it follows that 
\[\begin{array}{r l}
\frac{\lambda}{2}\binom{v}{2}-\tau-s+2& = \frac{1}{2}(\sum M +\sum N)-\tau-s+2\\[1mm]
 & \geq \frac{1}{2}(m_\tau+2(\tau-1)+ m_\tau-1+2(s-1) ) -\tau-s+2\\[1mm]
 & =  m_\tau-\frac{1}{2}.
\end{array}
\]
Therefore $\max(N,M)= m_\tau\leq \frac{\lambda}{2}\binom{v}{2}-\tau-s+2$ because $\frac{\lambda}{2}\binom{v}{2}-\tau-s+2$ is an integer. So by Theorem~\ref{Theorem:MultigraphDecomp} we can see that there exists an $(M,N)$-decomposition of $\lambda K_v$ which completes the proof.
\end{proof}

\vspace{0.3cm} \noindent{\large \bf Acknowledgements}

The author was supported by a Monash University Faculty of Science Postgraduate Publication Award.

\end{document}